\documentclass{amsart}
\usepackage[english]{babel}
\usepackage{amssymb}
\usepackage{hyperref}
\usepackage[initials]{amsrefs}
\usepackage{tikz-cd}
\usepackage{graphicx}
\usepackage{amssymb, mathrsfs, enumerate}
\usepackage{xcolor}
\usepackage{comment}
\newcommand{\bbN}{{\mathbb N}}
\newcommand{\bbQ}{{\mathbb Q}}
\newcommand{\bbR}{{\mathbb R}}

\newcommand{\bbZ}{{\mathbb Z}}

\newcommand{\defq}{\mathrel{\mathop:}=}

\newcommand{\im}{\operatorname{im}}

\newcommand{\vol}{\operatorname{vol}}

\newcommand{\mass}{\operatorname{mass}}

\newcommand{\jacobian}{\operatorname{Jac}}

\newcommand{\cell}{\ensuremath{\mathrm{cell}}}
\newcommand{\lip}{\ensuremath{\mathrm{lip}}}

\newcommand{\norm}[1]{{\left\lVert #1\right\rVert}}

\newtheorem*{mthm}{Main Theorem}
\newtheorem*{othm}{Theorem}
\newtheorem{theorem}{Theorem}[section]
\newtheorem{lemma}[theorem]{Lemma}

\theoremstyle{definition}

\newtheorem{defn}[theorem]{Definition}
\newtheorem{convention}[theorem]{Convention}

\newtheorem{example}[theorem]{Example}

\newtheorem{remark}[theorem]{Remark}

\begin{document}

\title[Waist inequalities in codimension two]{Uniform waist inequalities in codimension two for manifolds with Kazhdan fundamental group}

\author{Uri Bader}
\address{Weizmann Institute of Science, Israel}
\email{uri.bader@gmail.com}
\author{Roman Sauer}
\address{Karlsruhe Institute of Technology, Germany}
\email{roman.sauer@kit.edu}


\begin{abstract}
Let $M$ be a closed Riemannian manifold with Kazhdan fundamental group. 
It is well known that the Buser-Cheeger inequality yields a uniform waist inequality in codimension~$1$ for the finite covers of~$M$, which is basically another way of saying that the finite covers form an expander family. We show that 
the finite covers of~$M$ also satisfy a uniform waist inequality in codimension~$2$. 
\end{abstract}

\maketitle


\section{Introduction}

Gromov's waist inequality~\cites{gromov_filling,gromov_waist} for the sphere $S^n$ is a fundamental result in geometry. It says that the maximal volume of a fiber of a (generic) map from $S^n$ to $d$-dimensional Euclidean space is at least the $(n-d)$-dimensional (Hausdorff) volume of an equator sphere $S^{n-d}$.  

A family $\mathcal{F}$ of closed  $n$-dimensional Riemannian manifolds satisfies a \emph{uniform waist inequality in codimension~$d$}, where $1\le d<n$, if there is a constant~$C_{\mathcal{F}}>0$ such that for every $M\in \mathcal{F}$ and for every real analytic map $f\colon M\to \bbR^d$ there is $x\in \bbR^d$ such that the $(n-d)$-dimensional volume of the fiber of~$x$ satisfies 
\[ \vol_{n-d}\bigl(f^{-1}(\{x\})\bigr)\ge C_{\mathcal{F}}\cdot\vol(M).\]
By Whitney's theorem every smooth manifold carries a unique real analytic structure, and smooth maps can be arbitrarily well approximated by real analytic maps. 


A uniform waist inequality in codimension~$d$ is the Riemannian analog of a \mbox{$d$-di}mensional topological expander. The first explicit construction of expanders is by Margulis~\cite{margulis} and relies on the relative Kazhdan property. 
In \cite{Alon-Milman} Alon and Milman show that the Cayley graphs of finite quotients of a Kazhdan group form an expander family. The following Riemannian analog of this fact 
is a direct consequence of theorems by Buser and Cheeger relating the first eigenvalue of the Laplacian and the Cheeger constant~\cites{buser,cheeger}, in particular, of Buser's theorem.  
\begin{othm}[Corollary of the Cheeger-Buser inequality]
The family of finite covers of a closed Riemannian manifold with Kazhdan fundamental group (or only property~$\tau$) satisfies a uniform waist inequality in codimension~$1$.
\end{othm}

The surprising discovery in this paper is that the Kazhdan property gives an extra dimension for free. 

\begin{mthm}
The family of finite covers of a closed Riemannian manifold with Kazhdan fundamental group satisfies a uniform waist inequality in codimension~$2$.
\end{mthm}

Interesting examples of such manifolds are compact locally symmetric spaces of higher rank but, of course, every finitely presented Kazhdan group is the fundamental group of a closed smooth (in general, non-aspherical) manifold.  
If the group is infinite and residually finite then this manifold admits countably many finite covers.
Finitely presented, residually finite Kazhdan groups are in abundance,
the groups $\text{SL}_n(\mathbb{Z})$ for $n\geq 3$ being prominent examples.

In~\cite{fraczyk+lowe}, Fraczyk-Lowe prove that the entire family of finite volume  octonionic hyperbolic spaces (whose fundamental groups always have the Kazhdan property) satisfies a uniform waist inequality in codimension~$2$. While their methods are differential-geometric, we use methods from functional analysis and algebraic topology. 
Our proof has four main steps:

\begin{enumerate}
    \item\label{it: R-expansion} Use the Kazhdan property for proving $\mathbb{R}$-coboundary expansion.
    \item\label{it: Z-expansion} Upgrade $\mathbb{R}$-coboundary expansion to $\mathbb{Z}$-coboundary expansion.
    \item\label{it: isoperimetric} Derive isoperimetric inequalities from $\mathbb{Z}$-coboundary expansion.
    \item\label{it: waist} Derive waist inequalities from isoperimetric inequalities.
\end{enumerate}

The paper is decomposed into sections in accordance with the above list. 
The most innovative achievements of this paper are~\eqref{it: R-expansion} and~\eqref{it: Z-expansion}, the proof of the $\mathbb{Z}$-coboundary expansion. 

For~\eqref{it: isoperimetric}, we use Poincaré duality and 
a standard Federer-Fleming deformation argument 
to deduce from the $\mathbb{Z}$-coboundary expansion a \emph{uniform isoperimetric inequality} (see Section~\ref{sec:Lip to II}). 

For~\eqref{it: waist}, we implement a version of Gromov's filling method in~\citelist{\cite{gromov1}\cite{gromov2}\cite{gromov_filling}*{Appendix~F}}. The restriction to analytic maps or some other class of generic smooth maps is typical for this method but might not be necessary. We do not pursue this question here. The only use of analyticity is in Lemma~\ref{lem: real analytic maps} which appears in Gromov's work. We think that our particular implementation of the filling method is short and elegant. It has some similarity with the proof of Gromov's 
topological overlap theorem, which relies on uniform  $\mathbb{F}_2$-coboundary expansion in finite covers~\cites{gromov1, gromov2, overlap} but could be formulated with an appropriate uniform $\bbZ$-coboundary expansion. The proof of the Main Theorem, given in Section~\ref{sec: iso to waist}, is by contradiction and shares the following rough outline with other variations of the filling method: 
assuming smallness of fibers we deduce having a non-trivial top dimension cycle of arbitrarily small volume, which is absurd if we are working with a discrete group of coefficients. 

We now give an outline of~\eqref{it: R-expansion} and~\eqref{it: Z-expansion} and explain how the Kazhdan property gifts an extra codimension in the Main Theorem when compared to what we obtain from the Cheeger-Buser inequality. 
 The Kazhdan property provides some form of uniform coboundary expansion but a priori over $\bbR$ not $\bbZ$. Moreover, the Kazhdan property is about the cohomology in degree~$1$ with \emph{unitary coefficients} 
being Hausdorff (actually zero), while for the above mentioned uniform expansion we need to deal with \emph{$L^1$-coefficients} and with an additional Hausdorff property in degree~$2$. The latter is responsible for the surprising gain of an extra codimension.  
We take care of this in Section~\ref{sec: coboundary expansion}. Important ingredients are the Kazhdan property for $L^1$-Banach spaces~\cite{BGM} for the switch to $L^1$-coefficients and an ultraproduct technique adapted from~\cite{BS} for the Hausdorff property in degree~$2$.
A bridge between $\mathbb{Z}$-coefficients and $\mathbb{R}$-coefficients, using ideas from \emph{Integral Linear Programming}, is provided in Section~\ref{sec: total unimodularity}. See Theorem~\ref{thm:ZT} for a precise statement of the uniform $\mathbb{Z}$-coboundary expansion result.


\subsection*{Acknowledgment}
We wish to thank Mikolaj Fraczyk, Bo'az Klartag and Alex Lubotzky for various stimulating discussions. 
We thank the anonymous referee for a very thorough reading and many helpful suggestions that improved the presentation and clarity of this manuscript.
It is a special pleasure to thank Saar and Shaked Bader for going over a first draft of the paper, making several suggestions and pointing up relevant references.
U.B was funded by the BSF grant 713508.
R.S. was funded by the Deutsche Forschungsgemeinschaft (DFG, German Research 
Foundation) – 441426599.

\section{Coboundary expansion for $L^1$-spaces}
\label{sec: coboundary expansion}
The goal of this section is to prove that Kazhdan groups give rise to $1$-coboundary expanders with real coefficients.
Vector spaces are assumed to be defined over the reals.

\subsection{Abstract L-spaces}
Abstract L-spaces are an abstract 
characterization of spaces  $L^1(\Omega)$, where $\Omega$ is a measure space. 
The latter is called a \emph{concrete L-space}. 
The advantage of this notion, which we review below, is that it is well adapted to taking ultraproducts. 

An \emph{ordered (real) vector space} is a vector space $V$ with an order $\leq$ that is translation invariant and scale invariant for positive scalars.
The \emph{positive cone}, $V^+=\{x\mid x\geq 0\}$, determines the order: $x\leq y$ if and only if $y-x\in V^+$. 
If $(V,\leq)$ is a lattice in the sense of order theory, that is, every two elements $x,y$ have a \emph{meet} $x\wedge y$ and a \emph{join} $x \vee y$, we say that $(V,\leq)$ is a \emph{vector lattice}. In this case, the \emph{absolute value} of $x\in V$ is defined as $|x|=x \vee -x$
and an element $x_0\in V^+$ is called a \emph{unit} if $x_0\wedge y=0$ for $y\in V^+$ implies that $y=0$.

If $(V,\leq)$ is a vector lattice and $V$ carries a Banach space norm $\|\cdot \|$ such that $|x|\leq |y|$ implies $\|x\|\leq \|y\|$, we call $(V,\leq, \|\cdot \|)$ a \emph{Banach lattice}. 
The positive cone $V^+$ is closed and the lattice operations are continuous in this case. 

An \emph{abstract L-space} is a Banach lattice whose norm is additive on the positive cone, that is, $\|x+y\|=\|x\|+\|y\|$ for $x,y\in V^+$. 
We distinguish the class $\mathcal{L}$ consisting of abstract L-spaces that satisfy the additional axiom
\[ x\wedge y = 0  ~\Rightarrow~  \|x+y\|=\|x-y\|. \]

Let $\Omega$ be any measured space.
The (real) vector space $L^1(\Omega)$, endowed with the a.e pointwise order and the $L^1$-norm is in the class $\mathcal{L}$.
Such spaces are called \emph{concrete L-spaces}.
The following realization theorem is due to Kakutani.

\begin{theorem}[Kakutani~\cite{Kakutani}*{Theorems~4 and 7}]\label{thm: realization thm}
Every separable abstract L-space has a unit.
Every abstract L-space in the class $\mathcal{L}$ that has a unit is isometric and lattice isomorphic (by the same map) to a concrete L-space. 
\end{theorem} 

The following result is straightforward and the reason for our use of abstract L-spaces. See~\cite{heinrich} as a background reference for ultralimits of Banach spaces. 

\begin{lemma} \label{lem:calL}
   The class $\mathcal{L}$ is closed under taking ultraproducts and $\ell^1$-direct sums.
\end{lemma}

The following theorem is an important ingredient for the proof of coboundary expansion. 

\begin{theorem}[\cite{BGM}*{Corollary D}]\label{thm: L1 Kazhdan}
A countable group is Kazhdan if and only if all its isometric actions on concrete L-spaces have a fixed point. 
\end{theorem}


\begin{theorem}\label{thm:BGM}
    Let $\Gamma$ be a Kazhdan group and $V$ a Banach lattice in the class $\mathcal{L}$ on which $\Gamma$ acts by linear isometries. 
    Then $H^1(\Gamma,V)=0$.
\end{theorem}

\begin{proof}
Let $c\in H^1(\Gamma,V)$. By~\cite{BS}*{Lemma 3.6} there exists a $\Gamma$-invariant separable subspace $U\subset V$ such that~$c$ is in the image of the map $H^1(\Gamma,U)\to H^1(\Gamma,V)$ induced by the inclusion. 
    There is a $\Gamma$-invariant separable intermediate subspace $U\subset V_0\subset V$ that is a sublattice of $V$. Thus $V_0\in\mathcal{L}$. Let $c_0\in H^1(\Gamma,V_0)$ be a preimage of~$c$. 
    Using Theorem~\ref{thm: realization thm}, we regard $V_0$ as a concrete L-space. By Theorem~\ref{thm: L1 Kazhdan} and the interpretation of the fixed point property as  vanishing $1$-cohomology~\cite{bekka+valette}*{Lemma~2.2.6 on p.~77} we obtain that $c_0=0$. Hence $c=0$.
\end{proof}

\subsection{Coboundary expansion}

Let $\Gamma$ be a countable group and $V$ a Banach space on which $\Gamma$ acts by linear isometries. 
The cohomology groups  $H^i(\Gamma,V)$ are naturally endowed with a topological vector space structure~\cite{guichardet}*{Chapitre~III~\S 1} which comes from the Fréchet structure on the standard bar resolution. 
We can read off the topological structure from any countable CW-model~$X$ of the classifying space $B\Gamma$ as follows. 
Let $\tilde X$ denote the  universal cover of~$X$.
We consider the cellular cochain complex
\begin{align}\label{eq: cellular cochain complex}
\hom_{\bbZ\Gamma}\bigl( C^\cell_0(\tilde X), V\bigr)\xrightarrow{d^0_V} \hom_{\bbZ\Gamma}\bigl( C^\cell_1(\tilde X), V\bigr)\xrightarrow{d^1_V}\hom_{\bbZ\Gamma}\bigl( C^\cell_2(\tilde X), V\bigr) \to \cdots
\end{align}
whose cohomology is isomorphic to $H^*(\Gamma,V)$. 
For each~$i\in\bbN$ let $k_i\in\bbN\cup\{\infty\}$ be the number of $i$-cells in $X$.
By picking a representative in  each $\Gamma$-orbit of $i$-cells
we obtain an isomorphism 
\begin{align}\label{eq: ident}
\hom_{\bbZ\Gamma}\bigl( C^\cell_i(\tilde X), V\bigr) \cong \hom_{\bbZ\Gamma}\bigl( \oplus^{k_i}\bbZ\Gamma, V\bigr)\cong V^{k_i}
\end{align}
with the Fréchet space $V^ {k_i}$. The resulting Fréchet structure on every cochain group is independent of the choice of representatives. 
Under~\eqref{eq: ident} the differentials $d^i_V$ correspond to  $(k_{i+1}\times k_{i})$-matrices over $\mathbb{Z}\Gamma$ that we will denote by~$\delta^i$.
Note that the matrices $\delta^i$ are independent of $V$.
The dimension of~$\delta^i$ might be infinite, but every row has only a finite number of non-zero elements. So the multiplication with~$\delta^i$ is unambiguous.
It follows that each~$d^i_V$ is  continuous. Thus~\eqref{eq: cellular cochain complex} is a cochain complex of Fréchet spaces. 
Its cohomology groups inherit a subquotient topological vector space structure which coincides with the topological structure from the standard bar resolution. 

\begin{convention} \label{con: banach}
If $k_i<\infty$ we equip the space $V^{k_i}$ with the 
$\ell^1$-direct sum norm
\[ \|\bar{v}\|=\sum_{j=1}^{k_i}\|v_j\|,~~\bar{v}=(v_j)_{j=1}^{k_i}\in V^{k_i} \] 
and regard $\hom_{\bbZ\Gamma}\bigl( C^\cell_i(\tilde X), V\bigr)$ accordingly as a Banach space, using \eqref{eq: ident}.
The resulting norm on 
$\hom_{\bbZ\Gamma}\bigl( C^\cell_i(\tilde X), V\bigr)$ does not depend on the choice in the isomorphism~\eqref{eq: ident}. 
\end{convention}

In this section we will consider a \emph{finitely presented} Kazhdan group.
In this case the CW-complex $X$ could be chosen to have a finite $2$-skeleton,
so the differentials $d^0_V$ and $d^1_V$ become bounded maps between Banach spaces.

\begin{theorem}\label{thm:BGM2}
    Let $\Gamma$ be a finitely presented Kazhdan group and $V$ a Banach lattice in the class $\mathcal{L}$ on which $\Gamma$ acts by linear isometries. 
    Then $H^2(\Gamma,V)$ is Hausdorff as a topological vector space.
\end{theorem}

The proof uses an ultraproduct argument introduced in \cite[Theorem 3.11]{BS}.

\begin{proof}
Let us assume that $H^2(\Gamma,V)$ is not Hausdorff.
Let $X$ be a CW-model of $B\Gamma$ with a finite $2$-skeleton as above.
Via~\eqref{eq: ident} the chain complex \eqref{eq: cellular cochain complex} becomes 
\[ V^{k_0} \xrightarrow{\delta^0} V^{k_1} \xrightarrow{\delta^1} V^{k_2} \rightarrow\cdots \]
We consider each space in this chain complex as an $\ell^1$-direct sum of copies of $V$ and endow it with the corresponding norm.
The maps $\delta^0$ and $\delta^1$ are bounded maps between Banach spaces.
From our assumption we deduce that the image of $\delta^1$ is not closed.
It follows from (the easy direction of) the closed range theorem that the induced injection $V^{k_1}/\ker \delta^1\to V^{k_2}$ is not bounded from below, where the range is taken with the quotient Banach norm.
Therefore, we can find a sequence of unit vectors in the domain whose images converge to~$0$ in the range.
By taking a bounded (by, say, 2) sequence of preimages in $V^{k_1}$ we thus find a bounded sequence of elements $\bar{v}_n\in V^{k_1}$ such that $\lim_{n\to \infty} \delta^1(\bar{v}_n)=0$ in $V^{k_2}$ and for every cocycle $\bar{v}\in V^{k_1}$, $\|\bar{v}_n-\bar{v}\|\geq 1$ and in particular, for every $\bar{u}\in V^{k_0}$, $\|\bar{v}_n-\delta^0(\bar{u})\|\geq 1$.

We fix a non-principal ultrafilter $\omega$ and let $V_\omega$ be the corresponding ultrapower.
For $i\leq 2$, using the finiteness of $k_i$, we naturally identify the ultrapower of $V^{k_i}$ with $(V_\omega)^{k_i}$ and use the (now unambiguous) notation $V_\omega^{k_i}$.
We consider the element in $V_\omega^{k_1}$ which is represented by the sequence $(\bar{v}_n)$. It is a cocycle, that is, $\delta^1((\bar{v}_n))=0$ in $V_\omega^{k_2}$.
By Lemma~\ref{lem:calL}, $V_\omega$ is in the class $\mathcal{L}$, and we get by Theorem~\ref{thm:BGM} that $H^1(\Gamma,V_\omega)=0$.
It follows that $(\bar{v}_n)$ is a coboundary, that is there exists a sequence $\bar{u}_n\in V^{k_0}$ such that the element in $V_\omega^{k_0}$ represented by the sequence $(\bar{u}_n)$ satisfies $(\bar{v}_n)=\delta^0((\bar{u}_n))$.
This is a contradiction, as $\|\bar{v}_n-\delta^0(\bar u_n)\|\geq 1$ in $V^{k_1}$ for every $n$ yields $\|(\bar{v}_n)-\delta^0((\bar u_n))\|\geq 1$ in $V_\omega^{k_1}$. 
Hence $H^2(\Gamma,V)$ is Hausdorff.
\end{proof}

The next theorem associates with a choice of a CW-complex a certain constant of expansion in much the same way as the classical Kazhdan constant is associated with a choice of a finite generating set.

\begin{theorem}[Kazhdan constants]\label{thm: coboundary expansion for L-spaces}
Let $\Gamma$ be a finitely presented Kazhdan group and let $X$
be a CW-model of the classifying space $B\Gamma$ with a finite $2$-skeleton. 

There is a constant~$C>0$ with the following property. For every Banach lattice $V\in\mathcal{L}$ on which~$\Gamma$ acts by linear isometries, 
for each $i\in\{0,1\}$ and for every \[c\in \im \Bigl( \hom_{\bbZ\Gamma}\bigl( C^\cell_i(\tilde X), V\bigr)\xrightarrow{d^i_V}\hom_{\bbZ\Gamma}\bigl( C^\cell_{i+1}(\tilde X), V\bigr)\Bigr)\] 
there is an $i$-cochain~$b$ such that $d^i_V(b)=c$ and $\|b\|\le C\cdot\|c\|$. Here the norms are those introduced in~\ref{con: banach}. 
\end{theorem}

\begin{proof}
Fix $i\in\{0,1\}$. We assume by contradiction that 
there is a sequence $V_n$ of Banach lattices in the class $\mathcal{L}$ with linear isometric $\Gamma$-actions and non-zero cochains $c_n\in d^i_{V_n}\bigl(\hom_{\bbZ\Gamma}\bigl( C^\cell_{i}(\tilde X), V_n\bigr)\bigr)$ such that for each $b\in \hom_{\bbZ\Gamma}\bigl( C^\cell_{i}(\tilde X), V_n\bigr)$ with $d^i_{V_n}(b)=c_n$ we have $\norm{b}\ge n\norm{c_n}$. 
Upon rescaling, we assume that for every $n$, $\inf \{\|b\| \mid d^i_{V_n}(b)=c_n\}=1$.
In particular, $0<\|c_n\|\leq 1/n$.

For every $n\in\bbN$, we let $b_n\in \hom_{\bbZ\Gamma}\bigl( C^\cell_{i}(\tilde X), V_n\bigr)$ be a vector that satisfies $d^i_{V_n}(b_n)=c_n$, $1\leq \|b_n\|\leq 2$.
We let $U$ be the $\ell^1$-direct sum of the spaces $V_n$ and consider each vector $b_n$ as a vector in $U^{k_i}$, where $k_i$ is the number of $i$-cells in~$X$.
The space $U^{k_i}$ is in the class $\mathcal{L}$.
We conclude that $H^{i+1}(\Gamma,U)$ is Hausdorff, using Theorem~\ref{thm:BGM} for $i=0$ and Theorem~\ref{thm:BGM2} for $i=1$. Thus the image of $d^i_U$ is closed.
This contradicts the open mapping theorem, as witnessed by the sequence $b_n\in U^{k_i}$.
\end{proof}

Given a CW-complex $X$ with a finite $i$-skeleton, we identify the cellular cochain space $C_\cell^i(X,\bbR)$ with $\mathbb{R}^{k_i}$, where $k_i$ is the number of $i$-cells, and endow it with the corresponding $\ell^1$-norm.

\begin{theorem}[Coboundary expansion]\label{thm: real isoperimetric in degree 1}
Let $X$ be a connected CW-complex with Kazhdan fundamental group and a finite 2-skeleton. There is a constant~$C>0$ such that for every finite cover $\bar X\to X$, for every $i\in\{0,1\}$ and every cellular coboundary $c\in \im d^i\subset C_\cell^{i+1}(\bar X,\bbR)=\hom_\bbZ\bigl(C_{i+1}^\cell(\bar X),\bbR\bigr)$ there is $b\in C_\cell^i(\bar X,\bbR)$ such that $d^i(b)=c$ and
$\norm{b}\le C\cdot \norm{c}$. 
\end{theorem}

\begin{proof}
Let $\Gamma=\pi_1(X)$. The $2$-skeleton of $X$ is the $2$-skeleton of a model of the classifying space~$B\Gamma$, and we can apply Theorem~\ref{thm: coboundary expansion for L-spaces}. By covering theory, a finite cover $\bar X$ of $X$ corresponds to a subgroup $\Lambda<\Gamma$ of finite index. 
We equip $\ell^1(\Gamma/\Lambda)=\bbR[\Gamma/\Lambda]$ with the $\ell^1$-norm and for $i\leq 2$ we regard the cochain spaces $\hom_{\bbZ\Gamma}\bigl( C_i^\cell(\tilde X), \ell^1(\Gamma/\Lambda)\bigr)$ as Banach spaces using Convention~\ref{con: banach}. 
For the induced CW-structure on $\bar X$, the cellular cochain complex $C^\ast(\bar X,\bbR)$ with its $\ell^1$-norm is isometrically isomorphic to $\hom_{\bbZ\Gamma}\bigl( C_\ast^\cell(\tilde X), \ell^1(\Gamma/\Lambda)\bigr)$. Therefore the statement follows from Theorem~\ref{thm: coboundary expansion for L-spaces}. 
\end{proof}

\section{Expansion and Linear Programming}
\label{sec: total unimodularity}

The goal of this section is to prove an integral version of Theorem~\ref{thm: real isoperimetric in degree 1} (Theorem~\ref{thm:ZT}), which is about an optimization problem. To this end, we use some ideas from \emph{(integral) linear programming}. 
Linear Programming is concerned with maximizing a linear functional on a given convex polytope. Integral Linear Programming is concerned with finding a maximizer with integral coordinates.

Some notations are in order. An \emph{absolute value} on a nonzero commutative ring~$R$ is a function
$|\cdot|\colon R\to [0,\infty)$ satisfying $|0|=0$ and $|1|=1$ such that $|xy|=|x|\cdot |y|$ and $|x+y|\leq |x|+|y|$ for all $x,y\in R$.  
Given an absolute value on~$R$, we define for elements of a  finitely generated free module $R^n$:  
\[ \|(x_1,\ldots,x_n)\|=\sum_{i=1}^n |x_i|. \]

\begin{remark}
In the context of this paper we only care for the standard absolute value on $\mathbb{R}$ and its restrictions to $\mathbb{Q}$ and $\mathbb{Z}$.
However, a broader setup is useful for dealing e.g with the Hamming norm on $\mathbb{F}_2^n$ - compare the definition of expansion below to \cite[Definition 2]{overlap}.
Note that every nonzero commutative ring has an absolute value, namely the characteristic function of the complement of a prime ideal. 
\end{remark}

\begin{defn}
   For an $R$-linear transformation $A\colon R^n \to R^m$ we define its \emph{expansion constant} at $v\in A(R^n)$ as
   \[ \Xi(A,v) = \inf \bigl\{\alpha \geq 0 \mid \exists_{ u\in A^{-1}(\{v\})}~\|u\|\leq \alpha \cdot \|v\|\bigr\}. \]
   We define the \emph{expansion constant of $A$} by
   \[ \Xi(A)=\sup \bigl\{ \Xi(A,v) \mid v\in A(R^n) \bigr\}. \]
   Emphasizing the role of the ring $R$, we may write $\Xi_R(A)=\Xi(A)$. 
\end{defn}

Given a commutative $R$-algebra $S$ and viewing $A$ as a transformation $S^n\to S^m$, we might obtain a different expansion constant $\Xi_S(A)$.

\begin{lemma} \label{lem:ER<EZ}
        Let $A$ be an integer matrix. 
    Then $\Xi_{\bbR}(A) \leq \Xi_{\bbZ}(A)$.
    If the linear map $A$ is injective, then $\Xi_{\bbR}(A) = \Xi_{\bbZ}(A)$.
\end{lemma}

\begin{proof}
For every $u\in A(\mathbb{Z}^n)$ we clearly have $\Xi_\bbQ(A,u) \leq \Xi_\bbZ(A,u)$ with equality if $A$ is injective.
For every $v\in A(\mathbb{Q}^n)$ and $\beta \in \mathbb{Q}^\times$, we have $\Xi_\bbQ(A,v)=\Xi_\bbQ(A,\beta v)$ and we can pick $\beta$ such that $\beta v\in A(\mathbb{Z}^n)$. We thus obtain
   \[ \Xi_\bbQ(A)=\sup \{ \Xi_\bbQ(A,v) \mid v\in A(\mathbb{Z}^n) \}
   \leq \sup \{ \Xi_\bbZ(A,v) \mid v\in A(\mathbb{Z}^n) \}= \Xi_\bbZ(A) \]
   with equation for $A$ injective.
   Now $\Xi_\bbR(A)=\Xi_\bbQ(A)$ follows from continuity.
\end{proof}

The inequality in Lemma~\ref{lem:ER<EZ} might be strict if $A$ is not injective.

\begin{example}
    For the $(1\times 2)$-matrix $A=(1,2)$, $\Xi_\bbR(A)=1/2$ while $\Xi_\bbZ(A)=1$.
\end{example}

The rest of section is devoted to finding further sufficient conditions for the equation $\Xi_\bbZ(A)=\Xi_\bbR(A)$.
To this end, the following definition is important.

\begin{defn}
The matrix $A$ is said to be \emph{totally unimodular} if all of its minors are either $-1$, $0$ or $1$.
\end{defn}

In particular, every entry of a totally unimodular matrix is either $-1$, $0$ or $1$. 
Next we state a fundamental theorem in integral linear Programming by Hoffman-Kruskal.  
We use the notation $\bar{\mathbb{Z}}=\mathbb{Z}\cup\{\pm\infty\}$. 
The relation $\leq$ is understood coordinate-wise with the usual convention regarding $\pm\infty$.

\begin{theorem}[Hoffman-Kruskal, {\cite{Hoffman-Kruskal}*{Theorem 2}}] \label{thm:TU}
Let $A$ be an integer $m\times n$ matrix.
Then $A$ is totally unimodular if and only if for all $b,b' \in \bar{\mathbb{Z}}^n$ and $c,c' \in \bar{\mathbb{Z}}^m$, every face of the closed convex set 
\[ \bigl\{u\in \mathbb{R}^n \mid~b\leq u \leq b',~c\leq Au\leq c'\bigr\} \]
intersects $\mathbb{Z}^n$ non-trivially.
\end{theorem}

It is easy to see the stronger statement that, in the case of total unimodularity, every face is the convex hull of its integral points. 

\begin{lemma} \label{lem:TUE1}
    For a totally unimodular matrix $A$, $\Xi_\bbZ(A)= \Xi_\bbR(A)$.
\end{lemma}

An easy but important observation is that the $\ell^1$-norm coincides with a linear functional on each closed orthant. For $\epsilon=(\epsilon_1,\ldots,\epsilon_q)\in\{-1,1\}^q$, let
\[ O_\epsilon=\bigl\{x=(x_1,\ldots,x_q)\in\mathbb{R}^q\mid \epsilon_i x_i\geq 0\text{ for every }i\bigr\}. \]
On $O_\epsilon$, the $\ell^1$-norm coincides with the linear functional $x\mapsto\sum_i\epsilon_i x_i$.

\begin{proof}
Let $A$ be a totally unimodular $m\times n$ matrix. Let $v\in A(\mathbb{Z}^n)$. 
We claim that $\Xi_\bbZ(A,v) = \Xi_\bbR(A,v)$.
Set $\alpha=\Xi_\bbR(A,v)$.
Clearly, $\alpha \leq \Xi_\bbZ(A,v)$.
We will show that equality holds.
Let $u_0\in \mathbb{R}^n$ be an optimal vector such that $Au_0=v$ and $\|u_0\|=\alpha\cdot \|v\|$. 
Choose $\epsilon\in\{-1,1\}^n$ such that $u_0\in O_\epsilon$; if a coordinate of $u_0$ is zero, the corresponding sign may be chosen arbitrarily.
We consider the sets
\[ C=\bigl\{u\in O_\epsilon\mid Au=v\bigr\} \quad\mbox{and}\quad C_\alpha=\bigl\{u\in C\mid \|u\|=\alpha\cdot \|v\| \bigr\}.\]
The function $\|\cdot\|$ coincides with a linear functional on $C$, and $\alpha\cdot\|v\|$ is its minimum, by definition, which is attained at $u_0\in C$.
Moreover, $C$ is a closed convex set of the form appearing in Theorem~\ref{thm:TU}.
It follows that $C_\alpha$ is a face of $C$. By Theorem~\ref{thm:TU},
there exists an element $u_1\in C_\alpha \cap \mathbb{Z}^n$.
Hence $\Xi_\bbZ(A,v)\leq \alpha$ and $\Xi_\bbZ(A) \leq \Xi_\bbR(A)$.
The reverse inequality holds by Lemma~\ref{lem:ER<EZ}. 
\end{proof}

\begin{theorem} \label{thm:TUE2}
    Let $A$ and $B$ be integer matrices of sizes $(m\times n)$ and $(k\times m)$ such that the sequence 
    $\mathbb{R}^n \xrightarrow{A} \mathbb{R}^m \xrightarrow{B} \mathbb{R}^k$
    is exact. 
    If $A$ is totally unimodular, then $\Xi_\bbZ(A)= \Xi_\bbR(A)$ and $\Xi_\bbZ(B) = \Xi_\bbR(B)$.
\end{theorem}

\begin{proof}
  That $\Xi_\bbZ(A)= \Xi_\bbR(A)$ is proved in Lemma~\ref{lem:TUE1},
  so we are left to show $\Xi_\bbZ(B) = \Xi_\bbR(B)$.
  As in the proof of Lemma~\ref{lem:ER<EZ} this will follow from the claim that $\Xi_\bbZ(B,w) \leq \Xi_\bbR(B,w)$ for every $w\in B(\mathbb{Z}^m)$. 
  
  Let $w\in B(\mathbb{Z}^m)$ and set $\alpha=\Xi_\bbR(B,w)$.
  We pick $v_0\in\bbR^m$ and $v_1 \in\bbZ^m$ satisfying $Bv_0=Bv_1=w$ and  $\|v_0\|=\alpha\cdot \|w\|$.
  By exactness, we find $u_0\in \mathbb{R}^n$ such that $Au_0=v_0-v_1$.
  Choose $\epsilon\in\{-1,1\}^m$ such that $v_0\in O_\epsilon$; if a coordinate of $v_0$ is zero, the corresponding sign may be chosen arbitrarily.
  We consider the sets
\[ C=\bigl\{u\in \mathbb{R}^n\mid Au+v_1\in O_\epsilon\bigr\} ~\text{and}~ C_\alpha=\bigl\{u\in C\mid \|Au+v_1\|=\alpha\cdot \|w\| \bigr\}.\]
The function $u\mapsto \|Au+v_1\|$ coincides with a linear functional on $C$, and $\alpha\cdot\|w\|$ is its minimum, by definition, which is attained at $u_0\in C$.
Moreover, since $v_1$ is integral, $C$ is a closed convex set of the form appearing in Theorem~\ref{thm:TU}.

It follows that $C_\alpha$ is a face of $C$. By Theorem~\ref{thm:TU},
there exists an element $u_1\in C_\alpha \cap \mathbb{Z}^n$.
The element $Au_1+v_1\in \mathbb{Z}^m$ satisfies $\|Au_1+v_1\|=\alpha\cdot \|w\|$.
Hence $\Xi_\bbZ(B,w)\leq \alpha$.
\end{proof}

In view of the discussion above, finding criteria for total unimodularity of a matrix is desirable.
The following sufficient condition is well known.

\begin{lemma} \label{lem:TUcriterion}
    If $A$ is an integer matrix whose entries are in~$\{-1,0,1\}$ such that in each row there exists at most one entry $1$ and one entry $-1$, then $A$ is totally unimodular.
\end{lemma}

\begin{proof}
    Assuming the lemma is false we could find a counterexample $A$ of minimal size, which is clearly a square matrix of determinant other than $-1$, $0$ or $1$. The size of $A$ is greater than $(1\times 1)$ and removing a row with fewer than two non-zero elements leaves a totally unimodular matrix of smaller size, contradicting minimality. So all rows have both $1$ and $-1$. It follows that the vector $(1,1,\ldots,1)$ is in the kernel of $A$, thus $\det(A)=0$. This is a contradiction.
\end{proof}

\begin{theorem} \label{thm:dTU}
    For a finite CW complex $X$ let \[\bbR\xrightarrow{d^{-1}}C^0_\cell(X,\bbR)\xrightarrow{d^0} C^{1}_\cell(X,\bbR)\to\dots\] be the augmented cellular cochain complex. 
    Then $d^{-1}$ and $d^0$ are totally unimodular.  
\end{theorem}

We regard $d^k$ as a matrix with respect to the dual cellular basis.

\begin{proof}
The map $d^{-1}$ is represented by the column vector $(1,1,\ldots,1)$. 
The matrix representing $d^{0}$ has at each row either exactly one $1$ and one $-1$ or only zeroes. 
Hence $d^{-1}$ and $d^0$ are totally unimodular by Lemma~\ref{lem:TUcriterion}.
\end{proof}

\begin{remark}
  For $i\geq 1$, the matrix representing $d^{i}$ is an integer matrix which fails to be totally unimodular in general. The failure can be measured explicitly by the torsion in the relative homology of subcomplexes. See~\cite{DHK}. 
\end{remark}

\begin{theorem} \label{thm:integrality}
     Let $X$ be a finite CW complex with $H^1(X,\bbR)=0$. Then the cellular differentials satisfy 
     $\Xi_\bbZ(d^{0}) = \Xi_\bbR(d^{0})$ and $\Xi_\bbZ(d^{1}) = \Xi_\bbR(d^{1})$.
\end{theorem}

\begin{proof}
    This follows by combining Theorem~\ref{thm:dTU} with Theorem~\ref{thm:TUE2}.
\end{proof}

\begin{theorem} \label{thm:ZT}
     Let $X$ be a finite CW complex with Kazhdan fundamental group.  
     There is a constant~$C>0$ such that for every finite cover $\bar X\to X$, 
     the cellular differentials $d^k$ of~$\bar X$
     satisfy 
     $\Xi_\bbZ(d^{0}) = \Xi_\bbR(d^{0})\leq C$ and $\Xi_\bbZ(d^{1}) = \Xi_\bbR(d^{1})\leq C$.
\end{theorem} 

\begin{proof}
This follows by combining Theorem~\ref{thm:integrality} with Theorem~\ref{thm: real isoperimetric in degree 1}.    
\end{proof}

\section{Isoperimetric inequality for Lipschitz chains} \label{sec:Lip to II}
 In this section we convert the results about integral coboundary expansion of Section~\ref{sec: total unimodularity} into isoperimetric inequalities for Lipschitz chains. 
 
 A \emph{Lipschitz $d$-chain} in a Riemannian manifold~$M$ is a finite integral linear combination of Lipschitz singular $d$-simplices in~$M$. 
The \emph{$d$-dimensional volume} of a Lipschitz map $\sigma\colon \Delta^d\to M$ is defined as 
\[ 
\vol_d(\sigma)=\int_{\Delta^d}|\jacobian_\sigma(x)|dx.
\]
The \emph{mass} of a Lipschitz $d$-chain $c=\sum_i a_i\sigma_i$ written with distinct~$\sigma_i$ is defined as 
\[ 
\mass(c)=\sum_i |a_i|\cdot\vol_d(\sigma_i).
\]
The Riemannian manifold~$M$ satisfies the  
\emph{$C$-isoperimetric inequality in dimension~$d$} if for every Lipschitz chain $c\in C_d^\lip(M,\bbZ)$ that is a boundary there is $b\in C_{d+1}^\lip(M,\bbZ)$ with $c=\partial b$ such that 
\[ \mass (b)\le C\cdot\mass(c).\]

It is well known that one can convert statements about the expansion of the cellular chain complex into isoperimetric inequalities for Lipschitz chains via the following result, which builds on the Federer-Fleming deformation technique. 

\begin{theorem}[\cite{epstein-book}*{Theorem~10.3.3 on p.~223}]\label{thm: federer+fleming result}
Let $M$ be a closed Riemannian manifold which is endowed with a triangulation. Let $D\ge 1$ be a real number so that every simplex is $D$-bilipschitz equivalent to Euclidean standard simplex of the same dimension. Then there is a constant $E>0$ that only depends on the dimension of~$M$ and~$D$ such that if $c\in C_k^\lip(M,\bbZ)$ is a Lipschitz chain with $\partial c\in C_{k-1}^\cell(M,\bbZ)$, then there are chains $P(c)\in C_k^\cell(M,\bbZ)$ and $Q(c)\in C_{k+1}^\lip(M,\bbZ)$ such that: 
\begin{enumerate}
    \item $\partial Q(c)=c-P(c)$; in particular, $\partial c=\partial P(c)$,
    \item $\mass (P(c))\le E\cdot\mass(c)$,
    \item  $\mass (Q(c))\le E\cdot\mass(c)$.
\end{enumerate}
\end{theorem}

We obtain as a direct consequence:

\begin{theorem}\label{thm: from cellular to geometric}
Let $M$ be a closed Riemannian manifold which is endowed with a triangulation so that every simplex is $D$-bilipschitz equivalent to Euclidean standard simplex of the same dimension for some $D\ge 1$. Let $\partial^\cell_\ast$ denote the differentials of the cellular chain complex of the triangulation. 
Then $M$ satisfies the $C$-isoperimetric inequality in dimension~$k$ where $C>0$ only depends on $\Xi(\partial^ \cell_{k+1})$, $D$ and the dimension of~$M$. 
\end{theorem}

\begin{proof}
Let $c\in C_k^\lip(M,\bbZ)$ be a boundary. In particular, $c$ is a cycle. Choose $P(c)\in C_k^\cell(M,\bbZ)$ and $Q(c)\in C_{k+1}^\lip(M,\bbZ)$ according to Theorem~\ref{thm: federer+fleming result}. 
Since $c$ is a boundary and $c-P(c)=\partial Q(c)$, the chain $P(c)\in C_k^\cell(M,\bbZ)$ is a boundary in the Lipschitz singular chain complex. Since the inclusion of $C_\ast^\cell(M,\bbZ)$ into $C_\ast^\lip(M,\bbZ)$ induces isomorphisms in homology, the chain $P(c)$ has to be a boundary in $C_\ast^\cell(M,\bbZ)$ as well. So 
we can choose $b\in C_{k+1}^\cell(M,\bbZ)$ so that 
\[ \partial b=P(c)~\text{ and }~\norm b\le \Xi(\partial_{k+1}^\cell)\cdot \norm{P(c)}.\]
The assumption on the metric shape of the simplices implies that 
\[ \mass (b)\le \Xi(\partial_{k+1}^\cell)\cdot D^{2\dim M}\cdot\mass\bigl(P(c)\bigr).\]
We set $C\defq E\cdot \bigl( \Xi(\partial_{k+1}^\cell)\cdot D^{2\dim M}+1\bigr)$,
where~$E$ is the constant from Theorem~\ref{thm: federer+fleming result} that only depends on~$D$ and $\dim M$. 
The statement follows from 
 $\partial\bigl(b+Q(c)\bigr)=P(c)+c-P(c)=c$ and 
\begin{align*}
 \mass \bigl( b+Q(c)\bigr)&\le \Xi(\partial_{k+1}^\cell)\cdot D^{2\dim M}\cdot \mass \bigl(P(c)\bigr)+E\cdot\mass(c)\\
&\le E\cdot \bigl( \Xi(\partial_{k+1}^\cell)\cdot D^{2\dim M}+1\bigr)\cdot\mass (c).\qedhere
\end{align*}
\end{proof}

\begin{theorem}\label{thm: geometric isoperimetric inequality}
Let $M$ be a connected closed oriented $m$-dimensional Riemannian manifold with Kazhdan fundamental group. Then there is a constant~$C>0$ such that every finite cover $\bar M$ of $M$ satisfies $C$-isoperimetric inequalities in dimensions $m-2$ and $m-1$.  
\end{theorem}

\begin{proof}
We consider a triangulation of $M$. 
Let $\bar M$ be a finite cover. We consider the induced triangulation $T$ on $\bar M$
. 
Let $T_\cell$ denote the dual cell structure on~$\bar M$. 
To make explicit to which cell structure we refer when we take the cellular (co-)chain complex we write, for example,  $C^\cell_\ast(T, \bbZ)$ instead of $C^\cell_\ast(\bar M,\bbZ)$.  

The compactness of~$M$ implies that there is~$D>0$, independent of the choice of the finite cover, so that every simplex in $T$ is $D$-bilipschitz equivalent to the standard Euclidean simplex of the same dimension. 
Let $k\in \{0,1\}$. Let $d^k\colon C^k_\cell(T_\cell,\bbZ)\to C^{k+1}_\cell(T_\cell, \bbZ)$ be the differential of the cellular cochain complex of the dual cell structure. 
Similarly, let $\partial_k$ be the differential of the cellular chain complex of the simplicial structure. 
Poincaré duality yields a commutative square with vertical isomorphisms 
\[
\begin{tikzcd}
C^k_\cell(T_\cell,\bbZ)\ar[r,"d^k"] & C^{k+1}_\cell(T_\cell,\bbZ)\\
C_{m-k}^\cell(T,\bbZ)\ar[r,"\partial_{m-k}"]\ar[u,"\cong"]&C_{m-k-1}^\cell(T,\bbZ)\ar[u,"\cong"].
\end{tikzcd}
\]
The vertical isomorphism maps the simplicial basis to the dual cellular basis. Hence they are isometric with regard to the $\ell^1$-norms. 
In particular, 
\[\Xi\bigl(d^k\bigr)=\Xi\bigl(\partial_{m-k}\bigr).\]
By Theorem~\ref{thm:ZT}, these expansion constants are uniformly bounded over all finite covers for $k\in\{0,1\}$. 
Theorem~\ref{thm: from cellular to geometric} now implies the claimed isoperimetric inequalities in degrees $m-2$ and $m-1$.    
\end{proof}

\begin{remark}
   Compare the above with the result Kielak-Nowak~\cite{kielak+nowak}, showing that \emph{co}-boundary expansion in codimension~$2$ implies that the fundamental group is word-hyperbolic. 
\end{remark}

\section{From isoperimetric inequalities to waist inequalities}
\label{sec: iso to waist}

We show that isoperimetric inequalities lead to waist inequalities in the context of vanishing homology (see Theorem~\ref{thm: iso to waist}). At the end of 
this section we prove the Main  Theorem.

Let $f\colon M\to N$ be a smooth map between smooth manifolds, and let $\sigma\colon \Delta^k\to N$ be a smoothly embedded simplex. We say that $f$ \emph{intersects} $\sigma$ \emph{transversely} if $f$ intersects the interior $\sigma$ transversely and all its faces transversely. By definition, the map~$f$ intersects a $0$-simplex transversely if and only if the image point is a regular value of~$f$. In this case, 
the preimage $f^{-1}(\sigma)$ is a topological submanifold of dimension $\dim M-\dim N+k$ with boundary 
\[ \partial f^{-1}(\sigma)=f^{-1}(\partial \sigma).\]
In fact, it is a smooth submanifold with corners~\cite{meru}*{Proposition~3.3}. 
Recall the convention that the empty set is a manifold of every dimension. 
Moreover, the preimage is oriented if $M$ and $N$ are oriented. We say that $f$ is \emph{transverse to a triangulation} of~$N$ if $f$ intersects every simplex of the triangulation transversely. 

The following fact is probably true for arbitrary smooth maps~\cite{gromov_filling}*{$F_2''$ on~p.~134}. 

\begin{lemma}\label{lem: real analytic maps}
Let $f\colon M\to N$ be a real analytic map between smooth closed manifolds of dimensions $m\ge n$. Let $M$ be equipped with a smooth Riemannian metric. Let $\epsilon>0$. If $\vol_{m-n}(f^{-1}(\{x\}))<\infty$ for every $x\in N$, then there is a triangulation $T$ of $N$ such that $T$ is transverse to~$f$ and $\vol_{m-n+\dim\sigma}(f^{-1}(\sigma))<\epsilon$ for each $k$-simplex $\sigma$ in~$T$ with $k\ge 1$. 
\end{lemma}

The following theorem has an obvious generalization in case the homology vanishing and isoperimetric inequalities are available in more codimensions. 
The pleasant property about the Kazhdan property is that it implies the homology vanishing in codimension~$1$ for every finite cover via Poincaré duality.

\begin{theorem}\label{thm: iso to waist}
Let $M$ be a connected closed oriented $m$-dimensional Riemannian manifold satisfying
$H_{m-1}(M,\bbZ)=0$
and let $C>0$ be a constant such that $M$ satisfies the $C$-isoperimetric inequality in dimensions $m-1$ and $m-2$.
Then  for every $2$-dimensional closed oriented manifold $N$ and every real analytic $f\colon M\to N$ there exists a point $v\in N$ such that 
\begin{equation}\label{eq: explicit fiber estimate} \vol_{m-2}\bigl(f^{-1}(\{v\})\bigr)\ge \frac{1}{1+3C+12C^2}\cdot\vol_m(M).
\end{equation}
\end{theorem}

\begin{proof}
We may and will assume that $N$ is connected. Set $\epsilon_0\defq 1/(1+3C+12C^2)$. 
According to Lemma~\ref{lem: real analytic maps} let $T$ be a smooth oriented triangulation such that 
\begin{itemize}
    \item $f$ is transverse to~$T$, and  
    \item for every edge $e$ and every triangle $\Delta$ in $T$, $\vol_{m-1}(M_e)<\epsilon_0\cdot\vol_m(M)$ and $\vol_{m}(M_\Delta)<\epsilon_0\cdot\vol_m(M)$.
\end{itemize}
Here we denote the preimages of simplices by $M_\sigma\defq f^{-1}(\sigma)$ for $\sigma\in T$.
To prove the statement by contradiction, we assume that 
\begin{equation}\label{eq: contradiction hypothesis}
\vol_{m-2}(f^{-1}(v))<\epsilon_0\cdot\vol_m(M)
\end{equation}
for all $v\in N$, in particular for all vertices~$v$ of~$T$. 

In the following, all chain groups and homology groups are with integral coefficients. 
Next we construct a chain homomorphism 
\[ A_\ast\colon C_\ast^\cell(N)\to C_{m-2+\ast}^\lip(M)\]
that maps the fundamental class $[N]$ of $N$ to the one~$[M]$ of~$M$. The construction of $A_\ast$ is similar to the one in~\cite{meru}*{Proposition~3.6}.
For a simplex~$\sigma$ in~$T$, we write~$\sigma_i$ for its $i$-th face, so that $\partial \sigma=\sum_i (-1)^i\sigma_i$. 
For every vertex~$v$, choose a smooth triangulation of~$M_v$, and let $A_0(v)$ be its oriented simplicial fundamental chain, regarded as an integral Lipschitz chain in~$M$. Then
\[
\mass A_0(v)=\vol_{m-2}(M_v).
\]

Next, let $e$ be an oriented edge of~$T$. The chosen triangulations of $M_{e_0}$ and $M_{e_1}$ give a smooth triangulation of~$\partial M_e$. By the relative smooth triangulation theorem~\cite{munkres}*{Theorem~10.6}, it extends to a smooth triangulation of~$M_e$. Let $A_1(e)$ be its oriented simplicial fundamental chain. Then
\[
\partial A_1(e)=A_0(e_0)-A_0(e_1)
\quad\text{and}\quad
\mass A_1(e)=\vol_{m-1}(M_e).
\]

Finally, let $\Delta$ be an oriented triangle of~$T$. The triangulations chosen for its three edge preimages agree on their intersections and hence give a smooth triangulation of~$\partial M_\Delta$. After smoothing the corners of~$M_\Delta$, the same relative triangulation theorem extends it to a smooth triangulation of~$M_\Delta$. Let $A_2(\Delta)$ be its oriented simplicial fundamental chain. Then
\[
\partial A_2(\Delta)=A_1(\Delta_0)-A_1(\Delta_1)+A_1(\Delta_2)
\quad\text{and}\quad
\mass A_2(\Delta)=\vol_m(M_\Delta).
\]
Consequently, the $A_k$ define the desired chain homomorphism. The fundamental class of $N$ is represented by the sum of all oriented triangles. Its image under $A_2$ is a fundamental class for $M$ since for every $p\in M_\Delta$ its local orientation coincides with the one of $M_\Delta$.

We choose a vertex $v_0$ in $T$ as a base point. 
Now we define another chain map 
\[ B_\ast\colon C_\ast^\cell(N)\to C_{m-2+\ast}^\lip(M)\]
by setting $B_0(v)=A_0(v_0)$ for every vertex~$v$ and setting $B_k=0$ for $k>0$. 

Under the assumption~\eqref{eq: contradiction hypothesis} we will construct a chain homotopy $h_\ast\colon A_\ast\simeq B_\ast$. Because of $0\ne [M]=H_2(A_\ast)([N])=H_2(B_\ast)([N])$ this will give the desired contradiction. 

For $k=0$ and a vertex~$v$, the chain $A_0(v)-B_0(v)=A_0(v-v_0)=A_0(\partial w)=\partial A_1(w)$ is a boundary where $w$ is a $1$-chain connecting $v$ and~$v_0$. By~\eqref{eq: contradiction hypothesis} we obtain that \[\mass A_0(v-v_0)\le 2\cdot\epsilon_0\cdot\vol_m(M).\] 
The $C$-isoperimetric inequality in dimension $m-2$ implies that there  is $z_v\in C_{m-1}^\lip(M)$ with 
$\partial z_v=A_0(v-v_0)$ and $\mass z_v\le 2\cdot C\cdot \epsilon_0\cdot\vol_m(M)$. We define $h_0$ by setting $h_0(v)=z_v$ for each vertex~$v$. 

Let $e$ be an oriented edge in~$T$. 
Then 
\[ (A_1-B_1-h_0\circ\partial)(e)=A_1(e)-h_0(\partial e)\in C_{m-1}^\lip(M)\]
is a cycle. Because of $H_{m-1}(M)=0$ and the $C$-isoperimetric inequality in dimension~$m-1$ we can find $z_e\in C_m^\lip(M)$ such that $\partial z_e=(A_1-B_1-h_0\circ\partial)(e)$ and 
\begin{align*}
\mass (z_e)&\le C\cdot \mass \bigl(A_1(e)-h_0(\partial e)\bigr)\\
&< C\cdot \bigl( \epsilon_0\cdot \vol_m(M)+4\cdot C\cdot \epsilon_0\cdot \vol_m(M)\bigr)\\
&\le C\cdot (1+4C)\cdot \epsilon_0\cdot\vol_m(M).
\end{align*}
In the first line we used that $A_1(e)$ is a Lipschitz representative of the fundamental class of~$M_e$ and $\vol_{m-1}(M_e)<\epsilon_0\cdot\vol_m(M)$. 
We define $h_1$ by setting $h_1(e)=z_e$. 

Let $\Delta$ be a triangle  in~$T$. Then 
\[\bigl(A_2-B_2-h_1\circ\partial \bigr)(\Delta)=A_2(\Delta)-h_{1}(\partial \Delta)\in C_m^\lip(M)\]
is a cycle whose mass is bounded by 
\[
\mass\bigl(A_2(\Delta)-h_{1}(\partial \Delta)\bigr)<\epsilon_0\cdot (1+3C(1+4C))\cdot \vol_m(M)=\vol_m(M).
\]
Here we used that $A_2(\Delta)$ is a Lipschitz representative of the fundamental class of~$M_\Delta$ and $\vol_{m}(M_\Delta)<\epsilon_0\cdot\vol_m(M)$. 
As an integral Lipschitz cycle with mass less than $\vol_m(M)$ it is a boundary $\partial z_\Delta=A_2(\Delta)-h_{1}(\partial \Delta)$, and we can finish the construction of $h_\ast$ by setting $h_2(\Delta)=z_\Delta$. 
\end{proof}

If $\bar M\to M$ is a finite cover, then a waist inequality for $\bar M$ implies one with the same quality for $M$. So by passing to orientation covers we may assume that the manifold in the Main Theorem is oriented. 

\begin{proof}[Proof of the Main Theorem]
Let $M$ be a closed connected $m$-dimensional Riemannian manifold with Kazhdan fundamental group~$\Gamma$.

By Theorem~\ref{thm: geometric isoperimetric inequality} every finite cover $\bar M$ satisfies the $C$-isoperimetric inequality in dimensions $m-1$ and~$m-2$. 
The fundamental group $\bar \Gamma$ of a finite cover $\bar M$ is a finite index subgroup of~$\Gamma$ and is Kazhdan as well. By Poincaré duality we obtain that 
\[ H_{m-1}(\bar M,\bbZ)\cong H^1(\bar M,\bbZ)=H^1(\bar \Gamma,\bbZ).\]
Being a Kazhdan group implies that $H^1(\bar \Gamma,\bbR)=0$. By the universal coefficient theorem the first cohomology is torsion free. Hence $H_{m-1}(\bar M,\bbZ)=0$. 
Any real analytic map $\bar M\to \bbR^2$ can be regarded as a real analytic map $\bar M\to \bbR^2\hookrightarrow \bbR^2\cup \{\infty\}\cong S^2$. The Main Theorem now follows directly from Theorem~\ref{thm: iso to waist}. 
\end{proof}


\begin{bibdiv}
\begin{biblist}
\bib{meru}{article}{
   author={Alagalingam, Meru},
   title={Algebraic filling inequalities and cohomological width},
   journal={Algebr. Geom. Topol.},
   volume={19},
   date={2019},
   number={6},
   pages={2855--2898},
}

\bib{Alon-Milman}{article}{
    AUTHOR = {Alon, N.},
    AUTHOR = {Milman, V. D.},
     TITLE = {{$\lambda_1,$} isoperimetric inequalities for graphs, and
              superconcentrators},
   JOURNAL = {J. Combin. Theory Ser. B},
    VOLUME = {38},
      YEAR = {1985},
    NUMBER = {1},
     PAGES = {73--88},
}

\bib{BGM}{article}{
    AUTHOR = {Bader, U.},
    AUTHOR = {Gelander, T.},
    AUTHOR = {Monod, N.},
     TITLE = {A fixed point theorem for {$L^1$} spaces},
   JOURNAL = {Invent. Math.},
    VOLUME = {189},
      YEAR = {2012},
    NUMBER = {1},
     PAGES = {143--148},
}

\bib{BS}{article}{
      title={Higher Kazhdan property and unitary cohomology of arithmetic groups}, 
      author={Uri Bader},
      author={Roman Sauer},
      year={2023},
      eprint={2308.06517},
      archivePrefix={arXiv},
      url={https://arxiv.org/abs/2308.06517}, 
}
\bib{bekka+valette}{book}{
   author={Bekka, Bachir},
   author={de la Harpe, Pierre},
   author={Valette, Alain},
   title={Kazhdan's property (T)},
   series={New Mathematical Monographs},
   volume={11},
   publisher={Cambridge University Press, Cambridge},
   date={2008},
   pages={xiv+472},
}
\bib{buser}{article}{
   author={Buser, Peter},
   title={A note on the isoperimetric constant},
   journal={Ann. Sci. \'Ecole Norm. Sup. (4)},
   volume={15},
   date={1982},
   number={2},
   pages={213--230},
}
\bib{cheeger}{article}{
   author={Cheeger, Jeff},
   title={A lower bound for the smallest eigenvalue of the Laplacian},
   conference={
      title={Problems in analysis},
      address={Sympos. in honor of Salomon Bochner, Princeton Univ.,
      Princeton, N.J.},
      date={1969},
   },
   book={
      publisher={Princeton Univ. Press, Princeton, NJ},
   },
   date={1970},
   pages={195--199},
}
\bib{overlap}{article}{
   author={Dotterrer, Dominic},
   author={Kaufman, Tali},
   author={Wagner, Uli},
   title={On expansion and topological overlap},
   journal={Geom. Dedicata},
   volume={195},
   date={2018},
   pages={307--317},
}

\bib{DHK}{article}{
    AUTHOR = {Dey, Tamal K.},
    AUTHOR = {Hirani, Anil N.},
    AUTHOR = {Krishnamoorthy, Bala},
     TITLE = {Optimal homologous cycles, total unimodularity, and linear
              programming},
   JOURNAL = {SIAM J. Comput.},
  FJOURNAL = {SIAM Journal on Computing},
    VOLUME = {40},
      YEAR = {2011},
    NUMBER = {4},
     PAGES = {1026--1044},
}

\bib{epstein-book}{book}{
   author={Epstein, David B. A.},
   author={Cannon, James W.},
   author={Holt, Derek F.},
   author={Levy, Silvio V. F.},
   author={Paterson, Michael S.},
   author={Thurston, William P.},
   title={Word processing in groups},
   publisher={Jones and Bartlett Publishers, Boston, MA},
   date={1992},
   pages={xii+330},
}
\bib{fraczyk+lowe}{article}{
   author={Fraczyk, Mikolaj},
   author={Lowe, Ben},
   title={Minimal submanifolds and waists of locally symmetric spaces},
   note={Preprint}
}

\bib{gromov_filling}{article}{
   author={Gromov, Mikhael},
   title={Filling Riemannian manifolds},
   journal={J. Differential Geom.},
   volume={18},
   date={1983},
   number={1},
   pages={1--147},
}
\bib{gromov1}{article}{
   author={Gromov, Mikhail},
   title={Singularities, expanders and topology of maps. I. Homology versus
   volume in the spaces of cycles},
   journal={Geom. Funct. Anal.},
   volume={19},
   date={2009},
   number={3},
   pages={743--841},
}
\bib{gromov2}{article}{
   author={Gromov, Mikhail},
   title={Singularities, expanders and topology of maps. Part 2: From
   combinatorics to topology via algebraic isoperimetry},
   journal={Geom. Funct. Anal.},
   volume={20},
   date={2010},
   number={2},
   pages={416--526},
}

\bib{gromov_waist}{article}{
   author={Gromov, M.},
   title={Isoperimetry of waists and concentration of maps},
   journal={Geom. Funct. Anal.},
   volume={13},
   date={2003},
   number={1},
   pages={178--215},
}

\bib{guichardet}{book}{
   author={Guichardet, A.},
   title={Cohomologie des groupes topologiques et des alg\`ebres de Lie},
   language={French},
   series={Textes Math\'ematiques [Mathematical Texts]},
   volume={2},
   publisher={CEDIC, Paris},
   date={1980},
   pages={xvi+394},
}
		

\bib{heinrich}{article}{
   author={Heinrich, Stefan},
   title={Ultraproducts in Banach space theory},
   journal={J. Reine Angew. Math.},
   volume={313},
   date={1980},
   pages={72--104},
}

\bib{Hoffman-Kruskal}{incollection}{
    AUTHOR = {Hoffman, A. J.}, 
    AUTHOR = {Kruskal, J. B.},
     TITLE = {Integral boundary points of convex polyhedra},
 BOOKTITLE = {Linear inequalities and related systems},
    SERIES = {Ann. of Math. Stud.},
    VOLUME = {no. 38},
     PAGES = {223--246},
 PUBLISHER = {Princeton Univ. Press, Princeton, NJ},
      YEAR = {1956},

}

\bib{Kakutani}{article}{
    AUTHOR = {Kakutani, Shizuo},
     TITLE = {Concrete representation of abstract {$(L)$}-spaces and the
              mean ergodic theorem},
   JOURNAL = {Ann. of Math. (2)},
    VOLUME = {42},
      YEAR = {1941},
     PAGES = {523--537},
}

\bib{kielak+nowak}{misc}{
  author={Kielak, Dawid},
  author={Nowak, Piotr W.},
  title={Coboundary expansion and Gromov hyperbolicity},
  note={arXiv:2309.06215},
  date={2023},
}


\bib{margulis}{article}{
   author={Margulis, G. A.},
   title={Explicit constructions of expanders},
   language={Russian},
   journal={Problemy Pereda\v ci Informacii},
   volume={9},
   date={1973},
   number={4},
   pages={71--80},
}

\bib{munkres}{book}{
   author={Munkres, James R.},
   title={Elementary differential topology},
   note={Lectures given at Massachusetts Institute of Technology, Fall, 1961; revised edition},
   series={Annals of Mathematics Studies},
   volume={54},
   publisher={Princeton University Press, Princeton, N.J.},
   date={1966},
   pages={xi+112},
}

\end{biblist}
\end{bibdiv}
\end{document}